\documentclass[a4paper]{amsart}
\usepackage{enumerate,graphicx}
\theoremstyle{plain}
\newtheorem{thm}{Theorem}[section]
\newtheorem{rem}[thm]{Remark}
\newtheorem{pro}[thm]{Proposition}
\newtheorem{lem}[thm]{Lemma}
\newtheorem{cor}[thm]{Corollary}
\newtheorem{df}[thm]{Definition}
\newtheorem{exa}[thm]{Example}
\newtheorem{conj}[thm]{Conjecture}

\def\defin#1{{\it #1}}
\def\podloga#1{\left\lfloor #1\right\rfloor}
\def\sufit#1{\left\lceil #1\right\rceil}
\def\subP#1{P^{(#1)}}
\DeclareMathOperator{\m}{m}
\def\T{$^*$}
\def\maxN{9 }

\title{On path sequences of graphs}
\author{S\l{}awomir Bakalarski}
\address[S. Bakalarski]{Institute of Computer Science and Computational Mathematics\\
Faculty of Mathematics and Computer Science\\
Jagiellonian University\\
\L{}ojasiewicza 6\\
30-348, Krak\'{o}w}
\email{Slawomir.Bakalarski@uj.edu.pl}
\author{Jakub Zygad\l{}o}
\address[J. Zygad\l{}o]{Institute of Computer Science and Computational Mathematics\\
Faculty of Mathematics and Computer Science\\
Jagiellonian University\\
\L{}ojasiewicza 6\\
30-348, Krak\'{o}w}
\email{Jakub.Zygadlo@ii.uj.edu.pl}

\begin{document}
\begin{abstract}
A subset $S$ of vertices of a graph $G=(V,E)$ is called a $k$-path vertex cover if every path on $k$ vertices in $G$ contains at least one vertex from $S$. Denote by $\psi_k(G)$ the minimum cardinality of a $k$-path vertex cover in $G$ and form a sequence $\psi(G)=(\psi_1(G),\psi_2(G),\ldots,\psi_{|V|}(G))$, called the path sequence of $G$. In this paper we prove necessary and sufficient conditions for two integers to appear on fixed positions in $\psi(G)$. A complete list of all possible path sequences (with multiplicities) for small connected graphs is also given.
\end{abstract}
\maketitle

\section{Introduction}\label{secIntro}
Let $G$ be a graph and let $k$ be a positive integer. Following \cite{BKKS}, define a \defin{$k$-path vertex cover} ($k$-PVC for short) of $G$ as a subset $S$ of vertices of $G$ such that every path on $k$ vertices in $G$ has at least one vertex in common with $S$. A $k$-PVC is called \defin{minimum} if it has minimum cardinality among all $k$-path vertex covers of $G$. This minimum cardinality is denoted by $\psi_k(G)$ and called a \defin{$k$-path number} of $G$.
Path numbers generalize some well-known problems from graph theory, for example the cardinality of a minimum vertex cover of a graph $G$ equals $\psi_2(G)$, dissociation number of $G$ is equal to $|V|-\psi_3(G)$ (see \cite{BKKS}, \cite{KKS}) and in general values of $\psi_k(G)$ are exactly cardinalities of minimum vertex covers of $k$-uniform 'path hypergraph' $H$ built from $G$ (every path $P$ on $k$ vertices in $G$ gives rise to a hyperedge in $H$ containing vertices of $P$, see \cite{BKKS}).

We introduce the following definition:

\begin{df}
Let $G=(V,E)$ be a graph on $n$ vertices. The sequence of all path numbers, namely $$\psi(G)=(\psi_1(G),\psi_2(G),\ldots,\psi_n(G))$$
will be called a \defin{path sequence} of $G$.
\end{df}

The paper is devoted to investigation of the properties of path sequences.

\section{Elementary results}\label{secEleme}

Unless otherwise stated, in the following $G$ will denote a (simple, nonempty) graph on $n$ vertices and $k$ a positive integer satisfying $1\leq k\leq n$.

For the standard notations and definitions in graph theory we refer the reader to \cite{Bollobas}; here we only recall some extensively used notations. If $G=(V,E)$ is a graph and $S\subset V$, then $G[S]$ denotes the subgraph induced by $S$. Now for $v\in V$ and $e\in E$ we denote by $G-v$ the graph $G[V\setminus\{v\}]$ and by $G-e$ the graph $(V,E\setminus\{e\})$. We will also write $|G|$ for the number of vertices in $G$. By $P_n$, $C_n$ and $K_n$ we denote a path, a cycle and a complete graph on $n$ vertices respectively. A complete bipartite graph with partitions of size $a$ and $b$ will be denoted by $K_{a,b}$. The symbol $\simeq$ denotes graph isomorphism and by 'disjoint graphs' we mean vertex disjoint graphs. For a vertex $v$ of $G$ we denote by $d(v)$ the degree of $v$ and by $N(v)$ the neighbourhood of $v$ (the set of all vertices adjacent to $v$).

Let us first note that from the definition of path numbers one immediately gets $\psi_1(G)=n$, $\psi_k(G)\leq n-k+1$ (an arbitrary subset of $n-k+1$ vertices is a $k$-PVC) and that a path sequence is non-increasing, i.e. $\psi_1(G)\geq\psi_2(G)\geq\psi_3(G)\geq\ldots\geq\psi_n(G)\geq 0$.
An easy calculation gives path numbers for paths, cycles and complete graphs, namely
$\psi_k(P_n)=\podloga{\frac{n}{k}}$, $\psi_k(C_n)=\sufit{\frac{n}{k}}$ and $\psi_k(K_n)=n-k+1$ (see \cite{BJKST}).
We present values for complete bipartite graphs below.

\begin{pro}
Let $1\leq k\leq a+b$. Then:
$$\psi_k(K_{a,b})=\left\{\begin{array}{ll}a+b &\text{if }k=1,\\
\min\{a,b\}-\podloga{\frac{k}{2}}+1 &\text{for }1<k\leq 2\min\{a,b\}+1,\\
0 &\text{otherwise}.\end{array}\right.$$
\begin{proof}
Let us write $A$ and $B$ for partitions of $K_{a,b}$ with $|A|=a\leq b=|B|$. The case $k=1$ is clear. Assume that $1<k\leq a+b$ and take $p=a-\podloga{\frac{k}{2}}+1$. Since any path in $K_{a,b}$ alternates between $A$ and $B$, a path on $k$ vertices must have at least $\podloga{\frac{k}{2}}$ vertices in common with $A$. It follows that for $k>2a+1$ there is no path on $k$ vertices in $K_{a,b}$ and so $\psi_k(K_{a,b})=0$. So let $k\leq 2a+1$ and note that from the above reasoning an arbitrary set of $p$ vertices in $A$ is a $k$-PVC and consequently $\psi_k(K_{a,b})\leq p$. 

Now let $T$ be a subset of $A\cup B$ and $|T|=p-1$. To show that $\psi_k(K_{a,b})>p-1$ we will build a path on $k$ vertices 
disjoint from $T$. It suffices to find an arbitrary set of $\podloga{\frac{k}{2}}$ vertices in 
$A\setminus T$ and $\sufit{\frac{k}{2}}$ vertices in $B\setminus T$ (or vice versa). Note that there are at least 
$a-p+1=\podloga{\frac{k}{2}}$ vertices in $A\setminus T$ and also at least $b-p+1\geq\podloga{\frac{k}{2}}$ vertices in 
$B\setminus T$, since $b\geq a$. If $b>a$, then $b-p+1\geq\podloga{\frac{k}{2}}+1\geq\sufit{\frac{k}{2}}$ and the result follows. 
So we can assume that $a=b$. If $|A\setminus T|>\podloga{\frac{k}{2}}$, then $|A\setminus T|\geq\sufit{\frac{k}{2}}$ and we are 
done since $|B\setminus T|\geq\podloga{\frac{k}{2}}$. By symmetry, the only case left is $|A|=|B|$, $|A\setminus T|=|B\setminus 
T|=\podloga{\frac{k}{2}}$. But if $|A\setminus T|=\podloga{\frac{k}{2}}$, then $T\subset A$ and $B\cap T=\emptyset$, so 
$|B\setminus T|=|B|$. It follows that $|A\setminus T|=
|B\setminus T|=|B|=|A|$ and so $T=\emptyset$; consequently $p=1$ and $a=\podloga{\frac{k}{2}}$, so $k=2a$ or $k=2a+1$. It is 
easily verified that $\psi_{2a}(K_{a,a})=1$ and the value agrees with the formula given in the proposition; the case $k=2a+1$ is impossible since $k\leq a+b=2a$.
\end{proof}
\end{pro}

Let us also note the following useful lemma:

\begin{lem}\label{lem_v}
Let $G=(V,E)$ be a graph on $n\geq 2$  vertices, $k<n$ and $v\in V$. Then $\psi_k(G)\leq \psi_k(G-v)+1$.
Moreover, the following conditions are equivalent:
\begin{enumerate}
\item $\psi_k(G)=\psi_k(G-v)+1$,
\item $\exists S\subset V\!: S$ is a minimum $k$-PVC for $G$ and $v\in S$,
\item $\exists T\subset V\setminus\{v\}\!: T$ is a $k$-PVC for $G-v$ and $T\cup\{v\}$ is a minimum $k$-PVC for $G$.
\end{enumerate}
\begin{proof}
Let $U$ be a minimum $k$-PVC for $G-v$. Then clearly $U\cup\{v\}$ is a $k$-PVC for $G$ and so $\psi_k(G)\leq\psi_k(G-v)+1$.\\
Now suppose that (1) holds and $U$ is as above - then $U\cup\{v\}$ is a minimum $k$-PVC for $G$ and (2) follows with $S=U\cup\{v\}$. If (2) holds, then $T=S\setminus\{v\}$ is a $k$-PVC for $G-v$ (a path disjoint from $T$ in $G-v$ is disjoint from $S$ in $G$) and (3) clearly follows. Supposing that (3) holds gives $\psi_k(G)=|T|+1\geq \psi_k(G-v)+1\geq\psi_k(G)$ by the first part of the lemma, so $\psi_k(G)=\psi_k(G-v)+1$.
\end{proof}
\end{lem}

As a corollary we get:

\begin{cor}\label{cor_e}
Let $G=(V,E)$ be a graph and $e\in E$. Then $\psi_k(G)\leq \psi_k(G-e)+1$.
\begin{proof}
Let $e=uv$. Since $G-u$ is a subgraph of $G-e$, we apply previous lemma 
to obtain $\psi_k(G)\leq\psi_k(G-u)+1\leq\psi_k(G-e)+1$.
\end{proof}
\end{cor}

The following remark shows that there are no restrictions on the structure of a minimum $k$-PVC.

\begin{rem}
For any graph $H=(W,F)$, there exists a supergraph $G$ of $H$ such that $G[W]\simeq H$ and that $W$ is a minimum $k$-PVC for $G$.
\begin{proof}
We adapt the construction from the proof of \cite{BKKS}, Theorem 1. So let us replace each vertex $v$ of $H$ with a path $P^{(v)}=v-w^{(v)}_2-\ldots-w^{(v)}_k$ on $k$ vertices (paths for different $v$ are pairwise disjoint), leaving edges $F$ intact. Call the resulting graph $G$ and note that $G[W]\simeq H$. Now any $k$-PVC for $G$ 
must contain at least one vertex from each path $P^{(v)}$, so $\psi_k(G)\geq |W|$. But $W$ is clearly a $k$-PVC, so $\psi_k(G)=|W|$.
\end{proof}
\end{rem}

\section{Two element subsequences}\label{secTwoEl}

In this section we investigate relations between two path numbers $\psi_k(G)$ and $\psi_m(G)$ for an arbitrary graph $G$.
Let us start with the following example, showing that two elements of a path sequence must satisfy some additional conditions apart from the ones given in the previous section.

\begin{exa}
There is no graph $G$ satisfying $\psi_{10}(G)=2$ and $\psi_2(G)=5$.
\begin{proof}
Suppose that such a graph $G$ exists. Let $S$ be a minimum $2$-PVC for $G$ and $v\in S$. Note that at least one vertex from any edge in $G$ belongs to $S$. Take an arbitrary path $P$ in $G$ that avoids $v$. Since for every two consecutive vertices on $P$ at least one is from $S$ and $P$ has no more than 4 vertices in common with $S$ - it follows that $P$ is a path on at most $9$ vertices. This shows that $\{v\}$ is a $10$-PVC for $G$ and so $\psi_{10}(G)\leq 1$, a contradiction.
\end{proof}
\end{exa}

These additional necessary conditions are presented in the following proposition.

\begin{pro}\label{pro_war}
Let $1\leq m<k$ and $\psi_k(G)>0$. Then $\psi_m(G)\geq \psi_k(G)+\podloga{\frac{k}{m}}-1$.
\begin{proof}
We will proceed by induction on $n$ - the number of vertices in $G$. The result clearly follows for $n\leq 2$ and also for all graphs $G$ with $\psi_k(G)=1$ (including the case $k=n$), because 
we have $\psi_m(G)\geq \psi_m(P_k)=\podloga{\frac{k}{m}}$. So we can assume that $n>k\geq 2$ and $\psi_k(G)>1$. Let $S$ be a minimum $m$-PVC for $G$ and $v\in S$. By Lemma \ref{lem_v}, we get $\psi_m(G)=\psi_m(G-v)+1$ and $\psi_k(G-v)\geq\psi_k(G)-1>0$. By the induction hypothesis $\psi_m(G-v)\geq\psi_k(G-v)+\podloga{\frac{k}{m}}-1$ and consequently $\psi_m(G)=\psi_m(G-v)+1\geq\psi_k(G-v)+\podloga{\frac{k}{m}}\geq\psi_k(G)-1+\podloga{\frac{k}{m}}$.
\end{proof}
\end{pro}

\begin{rem}
Notice that the condition $\psi_k(G)>0$ in the above proposition cannot be omitted: take for example $G=K_{1,8}$ (a star on 9 vertices), $k=9$ and $m\in\{2,3,4\}$.
\end{rem}

\begin{rem}
Let $m<k$ and take $G$ equal to $s$ disjoint copies of $P_m$. Then clearly $\psi_m(G)=s$ and $\psi_k(G)=0$.
This shows that there exist graphs $G$ with $\psi_k(G)=0$ and an arbitrary value of $\psi_m(G)$.
\end{rem}

As the converse of Proposition \ref{pro_war} we show the following:

\begin{thm}\label{th_war}
Let $k$ be a positive integer and $1\leq m<k$. If two integers $p_k$, $p_m$ satisfy: $p_k>0$ and $p_m\geq p_k+\podloga{\frac{k}{m}}-1$, then there exists a (connected) graph $G$ such that $\psi_k(G)=p_k$ and $\psi_m(G)=p_m$.
\begin{proof}
Let $\podloga{\frac{k}{m}}=a$, i.e. $am\leq k<(a+1)m$. Take $p_k+a-1$ disjoint paths $\subP{1}, \subP{2},\ldots,\subP{p_k+a-1}$ on $2m-1$ vertices and let $\subP{i}=v^{(i)}_1-v^{(i)}_2-\ldots-v^{(i)}_{2m-1}$. Now add edges connecting vertices $v^{(i)}_x$ and $v^{(j)}_m$ for all $i\neq j$ and all $x$, i.e. $1\leq x\leq 2m-1$. Call the resulting graph $H$ and let $M=\{v^{(i)}_m: 1\leq i\leq p_k+a-1\}$ denotes the set of ``middle'' vertices of all $\subP{i}$ (see Fig. \ref{figH}). Since $\psi_m(\subP{i})=1$, it is easily seen that $M$ is a minimum $m$-PVC for $H$ and so $\psi_m(H)=|M|=p_k+a-1$.

\begin{figure}[ht]
\caption{An example graph $H$ for $m=3$ and $p_k+a-1=3$. Paths $\subP{i}$ are drawn horizontal, set $M$ is marked in black.}\label{figH}
\includegraphics[angle=90,scale=0.3]{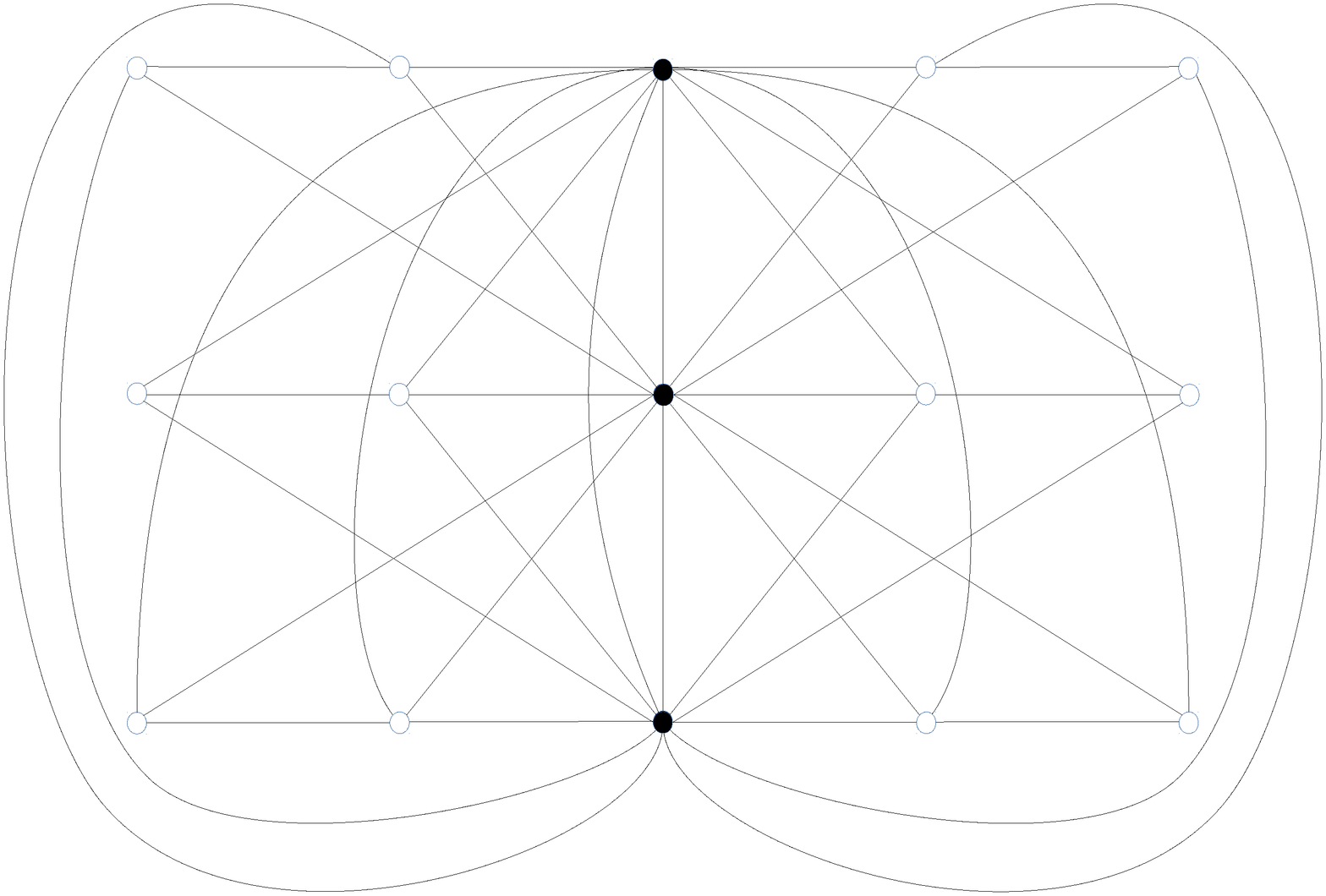}
\end{figure}
Let us now prove the following lemma concerning $H$ and $M$:

\begin{lem}\label{lem_sup} Any path on $k$ vertices in $H$ must contain at least $a$ vertices from $M$.
\begin{proof}
Suppose to the contrary that $W=w_1-\ldots-w_k$ is a path in $H$ and $|\{w_1,\ldots,w_k\}\cap M|=t<a$. Let us divide $W$ into consecutive fragments (subpaths) contained in paths $\subP{i}$, i.e. $W=W^1-\ldots-W^s$, where each $W^j$ is a (maximal) subpath of $W$ with all vertices in some fixed $\subP{i}$.
Note that it is possible for multiple $W^j$ to be subpaths of a single $\subP{i}$ and that any $W^j$ with at least $m$ vertices must contain some vertex from $M$.
Let us now define two sets of indices $j$:
$I_0=\{j: W^j\cap M=\emptyset\}$ and $I_1=\{j: W^j\cap M\neq\emptyset\}=\{j: W^j\text{ has exactly one vertex in common with }M\}$.
Clearly, the total number of $W^j$ equals $s=|I_0|+|I_1|=|I_0|+t$. Now we write $I_1$ as a disjoint sum of the following subsets:
$B=\{j: W^j\text{ has more than $m$ vertices}\}$, $U=\{j: W^j\text{ is a single vertex from }M\}$ and $R=I_1\setminus(B\cup U)=\{j\in I_1: 2\leq |W^j|\leq m\}$.
Counting the number of vertices in $W$ as a sum of the numbers of vertices in $W^j$ yields the following bound: 
\begin{align*}
|W|=& \sum_{j\in I_0\cup I_1}|W^j|=\sum_{j\in I_0\cup B\cup U\cup R}|W^j|\leq\\
\leq& |I_0|\cdot(m-1)+|B|\cdot(2m-1)+|U|\cdot 1+|R|\cdot m=\\
=& |I_0|\cdot(m-1)+|B|\cdot(2m-1)+|U|+(s-|I_0|-|B|-|U|)\cdot m=\\
=& (s+|B|-|U|)\cdot m+|U|-|I_0|-|B|
\end{align*}

We will now show the following claim: let $a<b$ be two integers such that the last (in order imposed by $W$) vertex in $W^{a}$ and the first vertex in $W^{b}$ are not in $M$. Then there exists an index $p\in U$ such that $a<p<b$. Indeed, by the construction of $H$ and $W^j$, since the last vertex in $W^{a}$ is not in $M$, the first one in $W^{a+1}$ must be in $M$. Analogously, the last vertex in $W^{b-1}$ must be in $M$. If $a+1\in U$ or $b-1\in U$, then we are done. If not, the last vertex in $W^{a+1}$ and the first one in $W^{b-1}$ are not in $M$ and we can proceed by induction on $b-a$ (the case $b-a=1$ being impossible and $b-a=2$ easily verified).

Since any $a,b\in I_0\cup B$ such that $a<b$ satisfy the hypothesis of the claim, we get $|U|\geq |I_0|+|B|-1$.
The bound for $|W|$ attains its maximum for the smallest possible $|U|$, that is $|U|=|I_0|+|B|-1$ and then we get
$|W|\leq (s-|I_0|+1)\cdot m-1=(t+1)\cdot m-1\leq am-1<k$, a contradiction that ends the proof.
\end{proof}
\end{lem}
We will show that $\psi_k(H)=p_k$. First note that by Lemma \ref{lem_sup} we get $\psi_k(H)\leq p_k$, since the set $S=\{v^{(i)}_m: 1\leq i\leq p_k\}$ is a $k$-PVC for $H$ as there are only $a-1$ vertices in $M\setminus S$.
Now let $T$ be an arbitrary set of no more than $p_k-1$ vertices from $H$. Without loss of generality we can assume that $T$ has no vertices in common with paths $\subP{1},\ldots,\subP{a}$ (recall that the number of $\subP{i}$ is $p_k+a-1$).
It is easy to observe that joining the paths $v^{(1)}_1-\ldots-v^{(1)}_{2m-1}$, $v^{(2)}_m-v^{(2)}_{m+1}-\ldots-v^{(2)}_{2m-1}, \ldots, v^{(a)}_m-v^{(a)}_{m+1}-\ldots-v^{(a)}_{2m-1}$ results in a path on $2m-1+(a-1) m= (a+1)m-1\geq k$ vertices. Consequently $T$ is not a $k$-PVC for $H$ and $\psi_k(H)>p_k-1$, so we must have $\psi_k(H)=p_k$.

Now we deal with the $m$-path number. Take $p_m-(p_k+a-1)$ (by assumption this number is non-negative) disjoint paths $Q^{(1)}, Q^{(2)},\ldots,Q^{(p_m-(p_k+a-1))}$ on $m$ vertices and let $Q^{(j)}=u^{(j)}_1-\ldots-u^{(j)}_m$. Connect all $Q^{(j)}$ to the vertex $v^{(1)}_m$ of $P^{(1)}$ by adding edges $u^{(j)}_1-v^{(1)}_m$ for all $j$. Resulting graph $G$ satisfies $\psi_k(G)=\psi_k(H)=p_k$, since $S=\{v^{(i)}_m: 1\leq i\leq p_k\}$ is a $k$-PVC for $G$. But also $\psi_m(G)=\psi_m(H)+(p_m-(p_k+a-1))=p_m$ because at least one vertex from each $P^{(i)}$ and each $Q^{(j)}$ must be included in the minimum $m$-PVC of $G$ and clearly $M\cup\{u^{(j)}_1:j=1,2,\ldots,p_m-(p_k+a-1)\}$ is a $m$-PVC for $G$.
\end{proof}
\end{thm}

\section{Path sequences for small graphs}\label{secPathS}

In this section we give some properties of path sequences concerning graphs with small number of vertices. First problem which arises naturally is the question whether equality of path sequences implies graph isomorphism. This is true for graphs with at most three vertices but false in general, as shown by the proposition below.

\begin{pro}\label{pro_isomorphic}
For any $n \geq 4$ there exist (connected) graphs $G$, $H$ on $n$ vertices such that $\psi(G)=\psi(H)$ but $G$ and $H$ are not isomorphic.
\end {pro}
\begin {proof}
Let $G$ be a graph and $v$ a vertex in $G$. By $G_{v,k}$ we understand a graph obtained from $G$ by adding $k$ new vertices $\{u_1,\ldots,u_k\}$ and edges $\{vu_1,vu_2,\ldots,vu_k\}$ to $G$. Now, let $u \in V(C_4)$ and let $v \in V(K_4-e)$ be of degree $3$ and consider graphs $G=(C_4)_{u,n-4}$ and $H=(K_4-e)_{v,n-4}$. Obviously $G$ and $H$ are not isomorphic but
$$\psi(G)=\psi(H)=\left\{\begin{array}{ll}
 (4,2,2,1) & \text { if } n=4, \\
 (n,2,2,1,1,0,\ldots,0) & \text { for } n \geq 5.
\end{array}\right.$$
\end {proof}

Before going further we state the following definition:

\begin {df}
Let $(p_1,\ldots,p_n)$ be a sequence of non-negative integers. Put
$$\begin{array}{rl}
\m(p_1,\ldots,p_n):=& \text{number of non-isomorphic connected graphs } G\\
                    & \text{on } n \text{ vertices such that } \psi(G)=(p_1,\ldots,p_n).\\
  \end{array}$$
We will call this number the \defin{path multiplicity of a sequence} $(p_1,\ldots,p_n)$. A sequence with nonzero path multiplicity will be called \defin{realisable}, i.e. $(p_1,\ldots,p_n)$ is realisable if there exists a connected graph $G$ with $\psi(G)=(p_1,\ldots,p_n)$. Moreover if at least one of the graphs $G$ satisfying $\psi(G)=(p_1,\ldots,p_n)$ is a tree, a bipartite graph, etc. we will say that the sequence is realisable by a tree, a bipartite graph, etc.
\end {df}

Tables 1 and 2 give realisable sequences and their path multiplicities for connected graphs on $n=5,6$ and $7$ vertices (for smaller $n$ all values are easily found "by hand"). These numbers were generated using a computer program written by the authors - source code is available at \cite{Prog}. The lists of non-isomorphic connected graphs and trees were obtained by the Mathematica package \cite{Math} and data from the web page \cite{McKay}. Note: sequences realisable by trees are marked with \T.
\noindent
\begin{table}[htp]
\caption {Path sequences for all connected graphs on 5 vertices (9 sequences, 21 graphs) and 6 vertices (20 sequences, 112 graphs).}
\begin{minipage}{0.4\textwidth}
{\small \begin {tabular}{|r|c|}
\hline 
\hfill Sequence\hfill\ & Multiplicity \\
\hline 
\T (5,1,1,0,0) & 1 \\
\T (5,2,1,1,0) & 2 \\
\T (5,2,1,1,1) & 1 \\
   (5,2,2,1,1) & 5 \\
   (5,3,1,1,1) & 2 \\
   (5,3,2,1,1) & 2 \\
   (5,3,2,2,1) & 5 \\
   (5,3,3,2,1) & 2 \\
   (5,4,3,2,1) & 1 \\ \hline
\end {tabular}}\hfill\ 
\end{minipage}
\begin{minipage}{0.55\textwidth}
{\small \begin {tabular}{|r|c||c|c|}
 \hline
\hfill Sequence\hfill\ & Mult. & Sequence & Mult.  \\
\hline
\T (6,1,1,0,0,0) &  1 & (6,3,3,2,1,1) &  5 \\ 
\T (6,2,1,1,0,0) &  2 & (6,3,3,2,2,1) & 14 \\
\T (6,2,2,1,0,0) &  1 & (6,4,2,1,1,1) &  4 \\
\T (6,2,2,1,1,0) & 10 & (6,4,2,2,1,1) &  1 \\
\T (6,3,1,1,1,0) &  3 & (6,4,2,2,2,1) &  8 \\
   (6,3,2,1,1,0) &  3 & (6,4,3,2,1,1) &  2 \\
\T (6,3,2,1,1,1) &  9 & (6,4,3,2,2,1) &  7 \\
   (6,3,2,2,1,0) &  1 & (6,4,3,3,2,1) &  9 \\
   (6,3,2,2,1,1) & 22 & (6,4,4,3,2,1) &  3 \\
   (6,3,2,2,2,1) &  6 & (6,5,4,3,2,1) &  1 \\ \hline
\end {tabular}}\hfill\ 
\end{minipage}
\end{table}
\noindent
\begin{table}[htp]
\caption {Path sequences for all connected graphs on 7 vertices (50 sequences, 853 graphs).}
\small{\begin {tabular}{|r|c||c|c||c|c|}
 \hline
\hfill Sequence\hfill\ & Multiplicity & Sequence & Mult. & Sequence & Mult. \\
\hline
\T (7,1,1,0,0,0,0) &  1 &(7,3,3,2,2,1,0) &  1 &(7,4,4,3,2,1,1) &  6 \\
\T (7,2,1,1,0,0,0) &  2 &(7,3,3,2,2,1,1) & 87 &(7,4,4,3,2,2,1) & 24 \\
\T (7,2,2,1,0,0,0) &  1 &(7,4,1,1,1,0,0) &  3 &(7,4,4,3,3,2,1) & 36 \\
\T (7,2,2,1,1,0,0) & 16 &(7,4,2,1,1,1,0) &  6 &(7,5,3,1,1,1,1) &  3 \\
\T (7,3,1,1,1,0,0) &  4 &(7,4,2,1,1,1,1) &  3 &(7,5,3,2,1,1,1) &  4 \\ 
\T (7,3,2,1,1,0,0) &  5 &(7,4,2,2,1,1,0) &  1 &(7,5,3,2,2,1,1) &  1 \\
\T (7,3,2,1,1,1,0) & 21 &(7,4,2,2,1,1,1) & 25 &(7,5,3,2,2,2,1) & 18 \\
\T (7,3,2,1,1,1,1) &  2 &(7,4,2,2,2,1,1) & 39 &(7,5,3,3,2,1,1) &  1 \\
   (7,3,2,2,1,0,0) &  1 &(7,4,3,1,1,1,0) &  1 &(7,5,3,3,2,2,1) &  8 \\
   (7,3,2,2,1,1,0) & 39 &(7,4,3,1,1,1,1) & 12 &(7,5,3,3,3,2,1) & 22 \\
   (7,3,2,2,1,1,1) &  4 &(7,4,3,2,1,1,0) &  3 &(7,5,4,3,2,1,1) &  2 \\
   (7,3,2,2,2,1,0) &  1 &(7,4,3,2,1,1,1) & 20 &(7,5,4,3,2,2,1) &  7 \\
   (7,3,2,2,2,1,1) &  9 &(7,4,3,2,2,1,1) & 69 &(7,5,4,3,3,2,1) & 19 \\
   (7,3,3,1,1,1,0) &  3 &(7,4,3,2,2,2,1) & 81 &(7,5,4,4,3,2,1) & 15 \\
   (7,3,3,1,1,1,1) &  8 &(7,4,3,3,2,1,1) & 46 &(7,5,5,4,3,2,1) &  3 \\
   (7,3,3,2,1,1,0) & 10 &(7,4,3,3,2,2,1) &129 &(7,6,5,4,3,2,1) &  1 \\
   (7,3,3,2,1,1,1) & 10 &(7,4,3,3,3,2,1) & 20 & - & -  \\ \hline
\end {tabular}}
\end{table}

From Tables 1 and 2 we can draw some observations. First of all notice that from the basic properties of path numbers it follows that
$\m(n,n-1,n-2,\ldots,1)=1$ (realisable by $K_n$) and that $\m(n,1,1,0,\ldots,0)=1$ (realisable by $K_{1,(n-1)}$, for $n\geq 3$). However there are many other path sequences with multiplicity one, for example $\m(5,2,1,1,1)=1$.

Proposition \ref{pro_isomorphic} shows that equality of path sequences for two graphs does not imply that they are isomorphic. The tables above show that we can even have that $\psi(G)=\psi(T)$ for some tree $T$ and some non-tree graph $G$.

However if $T_1, T_2$ are trees with $n<7$ vertices, then it follows from Tables 1 and 2 (and the number of non-isomorphic trees on $n$ vertices) that $\psi(T_1)=\psi(T_2) \iff T_1 \simeq T_2$. But this is also not true in general, as the following proposition shows.

\begin{pro}
For any $n \geq 7$ there exist trees $T_1$, $T_2$ on $n$ vertices such that $\psi(T_1)=\psi(T_2)$ but $T_1$ and $T_2$ are not isomorphic.
\begin{proof}
Assume first that $n=7$ and take the following trees:

\noindent
\begin {minipage}[t]{0.48 \textwidth}
\includegraphics [angle=90,scale=0.2] {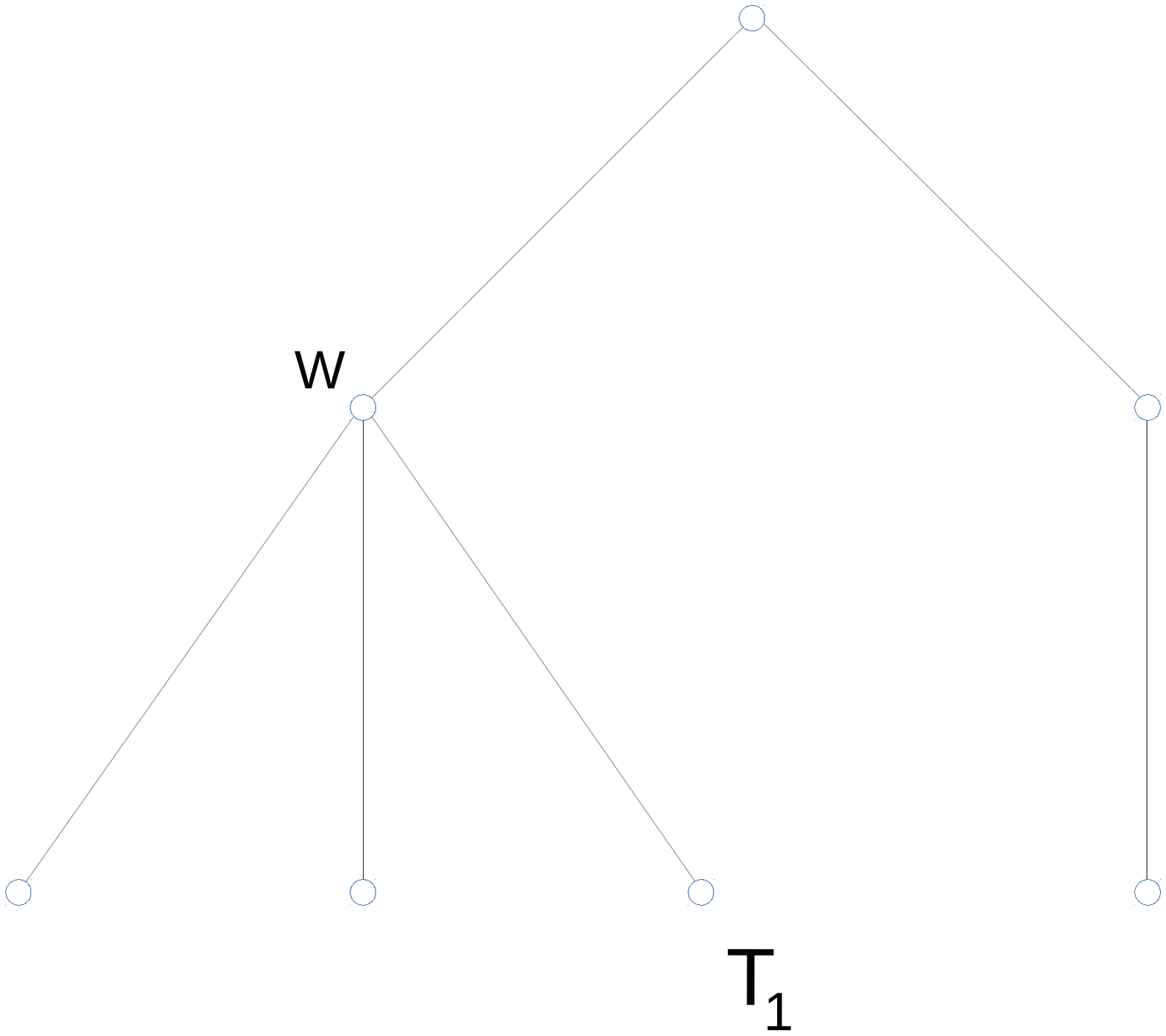}
\end {minipage}
\begin {minipage}[t]{0.48 \textwidth}
\includegraphics [angle=90,scale=0.2] {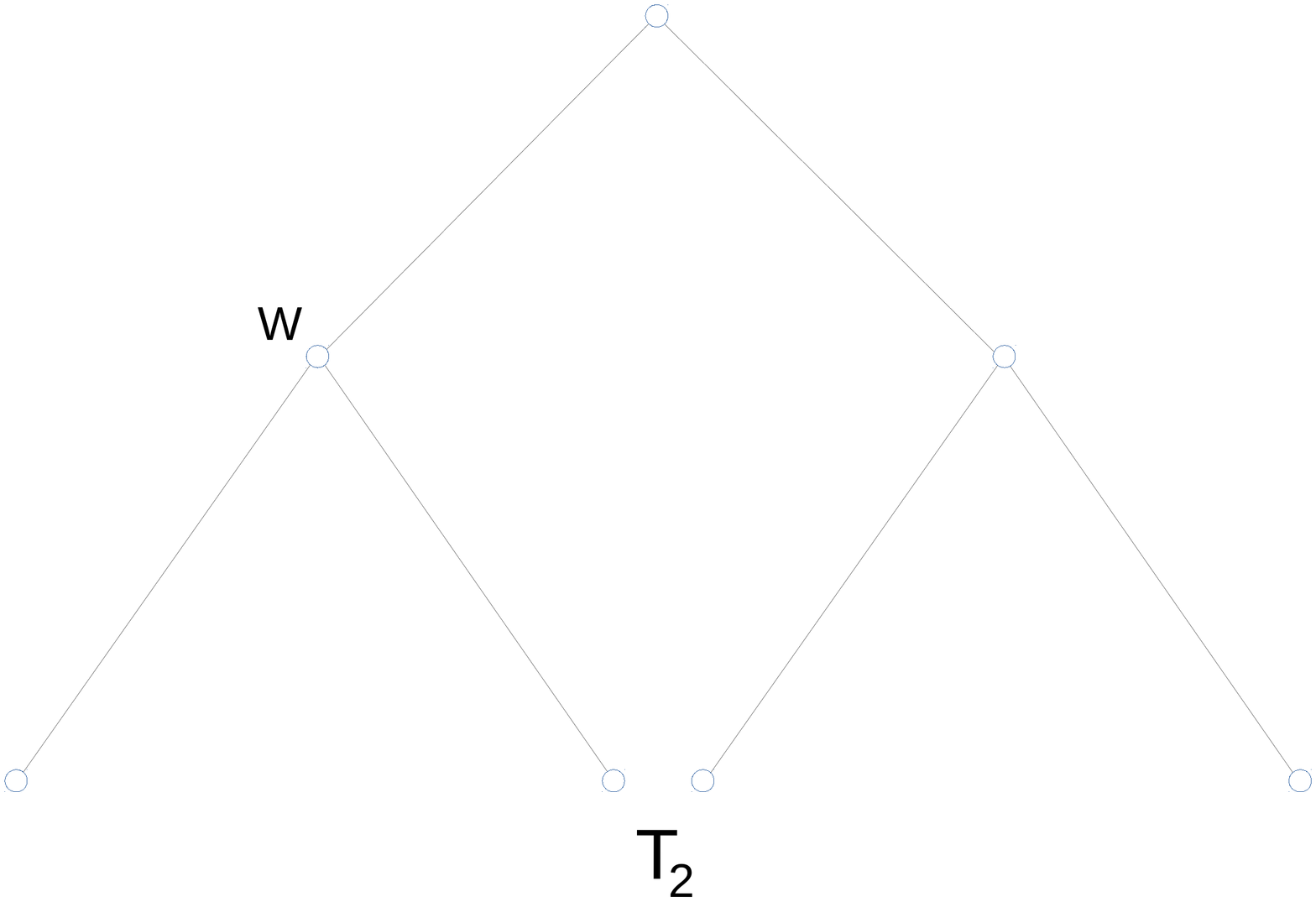}
\end {minipage}

\noindent
that are clearly not isomorphic and have path sequences equal to $(7,2,2,1,1,0,0)$.
For $n>7$ it suffices to attach additional vertices to $w$ (consider $(T_1)_{w,n-7}$ and $(T_2)_{w,n-7}$ in notations of Proposition \ref{pro_isomorphic}) to obtain non-isomorphic trees with path sequences $(n,2,2,1,1,0,\ldots,0)$.
\end{proof}
\end{pro}

By analysing the data for graphs with up to \maxN vertices (path sequences were calculated with the help of the computer program 
\cite{Prog}) we state the following conjecture concerning the existence of a Hamilton path in $G$ (that is clearly equivalent to the condition $\psi_n(G)=1$). According to our knowledge this conjecture has not been studied yet.

\begin {conj}\label{hip_2}
Let $G$ be a connected graph on $n\geq 2$ vertices. Then the following implication holds
$$\psi_{n-1}(G)=2 \Rightarrow \psi_n(G)=1$$
\end {conj}

By setting $k=n-1$ in the following remark, one observes that it is not necessary to formulate Conjecture \ref{hip_2} for disconnected graphs.

\begin{rem}\label{rem_hip2}
If $G$ is a graph on $n$ vertices such that $\psi_k(G)=n-k+1$ for some $2\leq k\leq n$, then $G$ is connected.
\begin{proof}
Let $n\geq 2$ and let $H_i$, $i\in I$ be the connected components of $G$. Since $\psi_k(H_i)\leq |H_i|-k+1$, 
we get $\psi_k(G)=\sum_{i\in I}\psi_k(H_i)\leq\sum_{i\in I}(|H_i|-k+1)=n-(k-1)\cdot|I|$. Now if $\psi_k(G)=n-k+1$, then $n-k+1\leq n-(k-1)\cdot|I|$. Since $k\geq 2$, equality is possible only for $|I|=1$, i.e. when $G$ is connected.
\end{proof}
\end{rem}

It is straightforward to see that if $G$ is a graph on $n$ vertices and $\psi_2(G)$ is maximum possible (i.e. $n-1$), then $G$ is necessarily isomorphic to $K_n$. It is not the case for $\psi_k(G)$ and $k>2$, however Conjecture \ref{hip_2} implies the following interesting property of path sequences.

\begin{thm}\label{ciekawa_wlasnosc}
Let $G$ be a graph on $n\geq 3$ vertices and $2\leq k<n$.
If Conjecture \ref{hip_2} holds for all connected graphs with at most $n$ vertices, then $\psi_k(G)=n-k+1$ implies $\psi_j(G)=n-j+1$ for all $j$ such that $k<j\leq n$.
\end{thm}
\begin{proof}
It is enough to prove the following claim for all graphs $G$ on $n\geq 3$ vertices and all $m$ such that $2\leq m<n$: if Conjecture \ref{hip_2} holds for all connected graphs with at most $n$ vertices, then $\psi_{m}(G)=n-m+1$ implies $\psi_{m+1}(G)=n-m$.

We proceed by induction on $n$. It easy to check the claim for $n=3$.
So assume that $n\geq 4$ and fix $m \in \{2,3,\ldots,n-1\}$. Let $G$ be a graph on $n$ vertices such that $\psi_m(G)=n-m+1$ and $\psi_{m+1}(G)=n-m-t$, with some $t\geq 0$. We need to prove that $t=0$. Observe that $G$ is connected by Remark \ref{rem_hip2} and if $m=n-1$, then the result follows by the validity of Conjecture \ref{hip_2}. So assume that $m<n-1$ and put $S$ to be a minimum $(m+1)$-PVC for $G$. There are two cases to consider:
\begin {enumerate}[(a)]
\item $S \neq \emptyset$. Choose $v \in S$ and let $G'=G-v$. Lemma \ref{lem_v} gives $n-m \leq \psi_m(G') \leq (n-1)-m+1$ and therefore $\psi_m(G')=n-m$. By the induction hypothesis $\psi_{m+1}(G')=n-m-1$, but due to Lemma \ref{lem_v} we obtain that $\psi_{m+1}(G')=\psi_{m+1}(G)-1=n-m-t-1$, so $t=0$.
\item $S=\emptyset$. This case cannot occur and the proof is as follows: let $w$ be any vertex of $G$,
put $G'=G-w$ and observe that by Lemma \ref{lem_v} we get $\psi_{m}(G')=n-m$, so by the induction hypothesis $\psi_{m+1}(G')=n-m-1$. But $0\leq \psi_{m+1}(G')\leq \psi_{m+1}(G)=0$ and consequently $n-m-1=0$, which is impossible since $m<n-1$.
\qedhere
\end {enumerate}
\end{proof}

We now present a lemma which will be useful in giving a direct proof of Conjecture \ref{hip_2} for graphs with no more than $7$ vertices. Note that the points (2)-(4) follow from Lemma 2.2 of \cite{BFSV}, however we include them here with a proof.

\begin {lem} \label{lem_3}
 Let $G=(V,E)$ be a connected graph on at least four vertices such that $V=\{p_1,\ldots,p_{n-1},q\}$, $N(q)=\{p_{i_1},\ldots,p_{i_t}\}$ and $p_1-\ldots-p_{n-1}$ is a path in $G$. If 
$\psi_{n-1}(G)=2$ and $\psi_n(G)=0$, then
\begin {enumerate}[(1)]
\item $d(p_1)\geq 2$ and $d(p_{n-1}) \geq 2$.
\item\label{lem_3_1} $p_1p_{n-1}\notin E$, $qp_1\notin E$, $qp_{n-1}\notin E$ and if $qp_i \in E$, then $qp_{i+1} \notin E$ for all $i\in\{2,\ldots,n-2\}$.
\item $p_1p_{i_j+1}\notin E$ and  $p_{n-1}p_{i_j-1} \notin E$ for all $j\in\{1,\ldots,t\}$.
\item $p_1p_{i_j-1} \notin E$, for all $j$ such that $i_j>\min\{i_1,\ldots,i_t\}$ and $p_{n-1}p_{i_j+1} \notin E$ for all $j$ such that $i_j<\max\{i_1,\ldots,i_t\}$.
\end {enumerate}
\end {lem}
\begin {proof}
To see (1) suppose to the contrary that $d(p_1)=1$.
We show that $S=\{p_2\}$ is a $(n-1)$-PVC for $G$. This follows from the fact that any path on $k$ vertices which avoids $p_2$ must also avoid 
$p_1$ and therefore $k\leq n-2$. The second case is proved analogously.

Now, the first part of (2) is obvious (in any case we get a path on all vertices in $G$, which contradicts $\psi_n(G)=0$). For the second, if such an $i$ exists, we have a path $p_1-\ldots-p_i-q-p_{i+1}-\ldots-p_{n-1}$.

As far as (3) is concerned, suppose first that $p_1p_{i_j+1} \in E$ for some $i_j$ such that $p_{i_j}\in N(q)$. 
Then we have the following path on $n$ vertices in $G$:
$q-p_{i_j}-p_{i_j-1}-\ldots-p_1-p_{i_j+1}-p_{i_j+2}-\ldots-p_{n-1}$.
If now $p_{n-1}p_{i_j-1} \in E$, then we get that the following path on $n$ vertices: 
$p_1-\ldots-p_{i_j-1}-p_{n-1}-\ldots-p_{i_j}-q$ exists.

To prove (4), let $r=\min\{i_1,\ldots,i_t\}$ and suppose that $p_1p_{i_j-1} \in E$. Then we have the following path on $n$ vertices in $G$:
$p_{n-1}-p_{n-2}-\ldots-p_{i_j}-q-p_{r}-p_{r-1}-\ldots-p_1-p_{i_j-1}-p_{i_j-2}-\ldots-p_{r+1}$. The second case follows by symmetry argument.
\end {proof}
 
\begin {cor}\label{cor_3}
 Let $G=(V,E)$ be a connected graph on at least four vertices such that $V=\{p_1,\ldots,p_{n-1},q\}$ and $p_1-\ldots-p_{n-1}$ is a path in $G$. If $\psi_{n-1}(G)=2$ and $\psi_n(G)=0$, then $$2 \leq d(q) \leq \sufit{\frac{n-3}{2}}$$
\end {cor}
\begin {proof}
Firstly notice that $d(q) \geq 2$, since if $d(q)=1$ and $qu\in E$, then $S=\{u\}$ is a $(n-1)$-PVC. To see this note that 
any path $P$ in $G$ that avoids $u$ must also avoid $q$ and so $|P|<n-1$.
As for the second inequality, it follows from Lemma \ref{lem_3}.(\ref{lem_3_1}) since for every $i \in \{2,\ldots,n-2\}$ at most one of the edges $qp_i$, $qp_{i+1}$ is in $E$.
\end {proof}

The above facts allow us to give a direct proof of Conjecture \ref{hip_2} for graphs with no more than $7$ vertices. Our reasoning is "by considering cases" -- unfortunately we were unable to find a more general approach.

\begin {thm}
Let $G=(V,E)$ be a connected graph on $n$ vertices, with $2\leq n\leq 7$. Then Conjecture \ref{hip_2} holds for $G$, i.e.
$$\psi_{n-1}(G)=2 \Rightarrow \psi_n(G)=1$$
\end {thm}

\begin {proof}
The theorem holds true for $n=2$ and $n=3$, with complete graphs $K_2$ and $K_3$ being the only cases to verify. So we can assume that $n\geq 4$.

It is sufficient to prove that the existence of a connected graph that satisfies $\psi_{n-1}(G)=2$ and $\psi_n(G)=0$ leads to a contradiction. Throughout we consider a graph $G=(V,E)$ with $V=\{p_1,\ldots,p_{n-1},q\}$ and assume that $p_1-\ldots-p_{n-1}$ is a path in $G$.

Notice that if $n=4$ or $n=5$, then $\sufit{\frac{n-3}{2}}<2$ and Corollary \ref{cor_3} gives $2\leq d(q)<2$, a contradiction. So we can assume that $n=6$ or $n=7$. Note that in both cases we get $d(q)=2$ by Corollary \ref{cor_3}.

Now let $G$ be a connected graph with $6$ vertices such that $\psi_5(G)=2$, $\psi_6(G)=0$ and $d(q)=2$. Due to Lemma \ref{lem_3} 
we only need to consider the case when $qp_2,qp_4 \in E$ - but then the set $S_5=\{p_2\}$ is a $5$-PVC. To see this let us assume 
that there exists a path $P$ on 5 vertices $p_1,p_3,p_4,p_5,q$ (in any order). Using again Lemma \ref{lem_3} we obtain that $d(p_1)=2$,$p_1p_4 \in E$ and $p_3p_5 \notin E$.
It follows that $p_2p_3$ and $p_3p_4$ are the only two edges containing $p_3$. Consequently $P$ must start with $p_3$ and it is easy to see that we cannot build a path avoiding $p_2$ longer than $p_3-p_4-x$ where $x \in \{q,p_1,p_5\}$, which contradicts $P$ having 5 vertices.

Let us now assume that $n=7$ and there exists a graph $G$ such that $\psi_6(G)=2$ and $\psi_7(G)=0$ with $d(q)=2$. Because of Lemma \ref{lem_3} and symmetries we only need to consider two cases:  $qp_2,qp_5 \in E$ and $\; qp_2, qp_4 \in E$. Assume the first case - by Lemma \ref{lem_3} we get that $d(p_1)=2$ and $p_1p_5 \in E$. But now $S_6=\{p_2\}$ is a $6$-PVC for $G$: this follows from the fact that there is no path on $6$ vertices which avoids $p_2$ in $G$. Indeed, if such a path exists, then it is of the form $q-p_5-x_1-x_2-x_3-x_4$ where $x_i \neq p_1$ for $i=1,2,3,4$ -- a contradiction.

Let us now proceed with the case $qp_2, qp_4 \in E$. By Lemma \ref{lem_3} we must have that $d(p_1)=2$ and $p_1p_4 \in E$. But then again $S_6=\{p_2\}$ is a $6$-PVC for $G$. To see this, notice that we only need to consider paths of the form $q-p_4-x_1-x_2-x_3-x_4$. If such a path omits $p_2$, then we cannot have $x_1=p_1$ and so we get $x_i \in \{p_3,p_5,p_6\}$ for $i=1,2,3,4$ -- a contradiction.
\end {proof}

As a consequence of the above and Theorem \ref{ciekawa_wlasnosc} we get

\begin{cor}
Let $G$ be a graph on $n$ vertices with $3\leq n\leq 7$ and let $2\leq k<n$. If $\psi_k(G)=n-k+1$, then $\psi_j(G)=n-j+1$ for all $j$ such that $k<j\leq n$.
\end{cor}

\end{document}